\documentclass[10pt]{article}

\usepackage[english]{babel}
\usepackage{amsmath}
\usepackage{placeins}
\usepackage{amssymb}
\usepackage{latexsym}
\usepackage{amsmath}
\usepackage{amsthm}
\usepackage{enumerate}

\newtheorem{theorem}{Theorem}
\newtheorem{predl}{Theorem}
\newtheorem{lemma}{Lemma}

\begin{document}

\title{Marcinkiewicz--Zygmund Strong Law of Large Numbers \\ for Pairwise i.i.d. Random Variables}

\author{Valery Korchevsky\thanks{Saint-Petersburg State University of Aerospace Instrumentation, Saint-Petersburg. \endgraf E-mail: \texttt{valery\_ko@list.ru} }}

\date{}

\maketitle

\begin{abstract}
It is shown that the Marcinkiewicz--Zygmund strong law of large numbers holds for pairwise independent identically distributed random variables. It is proved that if $X_{1}, X_{2}, \ldots$ are pairwise independent identically distributed random variables such that $E|X_{1}|^p < \infty$ for some $1 < p < 2$, then $(S_{n}-ES_{n})/n^{1/p} \to 0$ a.s. where $S_{n} = \sum_{k=1}^{n} X_{k}$.
\end{abstract}

\bigskip

\noindent \textbf{Keywords:} strong law of large numbers, pairwise independent random variables, identically distributed random variables.

\bigskip

\bigskip

\noindent {\Large{\textbf{1. Introduction.}}}

\medskip

\noindent Let $\{X_{n}\}_{n=1}^{\infty}$ be a sequence of independent identically distributed random variables. There are two famous theorems on the strong law of large numbers for such a sequence: The Kolmogorov theorem and the Marcinkiewicz--Zygmund theorem (see e.g. Lo\`{e}ve~\cite{Loeve77}). Let $S_{n} = \sum_{k=1}^{n} X_{k}$. By Kolmogorov's theorem, there exists a constant $b$ such that $S_{n}/n \to b$ a.s. if and only if $E|X_{1}| < \infty$; if the latter condition is satisfied then $b = EX_{1}$.

Now we state the Marcinkiewicz--Zygmund theorem:

\begin{predl}\label{P101}

Let $\{X_{n}\}_{n=1}^{\infty}$ be a sequence of independent identically distributed random variables. If $0 < p < 2$ then the relation $E|X_{1}|^p < \infty$ is equivalent to the relation

\begin{equation}\label{e11}
   \frac{S_{n}-nb}{n^{1/p}} \to 0 \qquad \mbox{a.s.}
\end{equation}

\noindent Here $b=0$ if $0 < p < 1$, and $b = EX_{1}$ if $1 \leqslant p < 2$.

\end{predl}

The aim of this work is to show that Theorem~\ref{P101} remains true if we replace the independence condition by the condition of pairwise independence of random variables $X_{1}, X_{2}, \ldots$

Etemadi~\cite{Etem81} proved the Kolmogorov theorem under the pairwise independence assumption instead of the independence condition. Sawyer~\cite{Sawyer66} showed that if $0 < p < 1$ then the condition $E|X_{1}|^p < \infty$ implies $S_{n}/n^{1/p} \to 0$ a.s. without any independence condition. Petrov~\cite{Petr96} proved that if $0 < p < 2$ then relation~\eqref{e11} (with $b=0$ or $EX_{1}$ according as $p < 1$ or $p \geqslant 1$) implies that $E|X_{1}|^p < \infty$ assuming pairwise independence.

In the present work we shall prove that if $1 < p < 2$ then the condition $E|X_{1}|^p < \infty$ implies $(S_{n}-ES_{n})/n^{1/p} \to 0$ a.s. under the pairwise independence assumption. There are a number of papers that contain results on the strong law of large numbers for sequences of pairwise independent identically distributed random variables. See Choi and Sung~\cite{ChoiSung85}, Li~\cite{Li88}, Martikainen~\cite{Mart95}, Sung~\cite{Sung97, Sung14} (recent work~\cite{Sung14} contains more detailed review). However, results in these papers do not generalize Theorem~\ref{P101} to sequences of pairwise independent random variables.

\bigskip

\noindent {\Large{\textbf{2. Main results.}}}

\medskip

\noindent The aim of this paper is to prove the following result:

\begin{theorem}\label{T1}
Let $\{X_{n}\}_{n=1}^{\infty}$ be a sequence of pairwise independent identically distributed random variables. If $E|X_{1}|^p < \infty$ where $1 < p < 2$, then

\begin{equation}\label{e1001}
   \frac{S_{n}-ES_{n}}{n^{1/p}} \to 0 \qquad \mbox{a.s.}
\end{equation}

\end{theorem}

If we combine Etemadi's, Sawyer's, and Petrov's results mentioned in the previous section with Theorem~\ref{T1}, we get a generalization of the Marcinkiewicz--Zygmund theorem (Theorem~\ref{P101}):
{\sloppy

}

\begin{theorem}\label{T2}
Let $\{X_{n}\}_{n=1}^{\infty}$ be a sequence of pairwise independent identically distributed random variables. If $0 < p < 2$ then the relation $E|X_{1}|^p < \infty$ is equivalent to relation~\eqref{e11}.

\end{theorem}

\bigskip

\noindent {\Large{\textbf{3. Proof of Theorem~\ref{T1}.}}}

\medskip

\noindent To prove our main result we need the following lemmas.

\begin{lemma}\label{Lem201}
Let $\{X_{n}\}_{n=1}^{\infty}$ be a sequence of identically distributed random variables. If $E|X_{1}|^p < \infty$ where $1 < p < 2$, then

\begin{equation}\label{e201}
\frac{\sum_{i=1}^{n} |X_{i}| \mathbb{I}_{\{|X_{i}| > n^{1/p}\}}}{n^{1/p}} \to 0 \qquad \mbox{a.s.}
\end{equation}

\end{lemma}

\begin{proof}
Let $U_{n} = |X_{n}|^{p} \mathbb{I}_{\{|X_{n}| > n^{1/p}\}}$, $n \geqslant 1$. Note that condition $E|X_{1}|^p < \infty$ is equivalent to the relation

\begin{equation}\label{e202}
\sum_{n=1}^{\infty} P(|X_{1}| > n^{1/p}) < \infty.
\end{equation}

\noindent Thus, we have

\begin{equation*}
\sum_{n=1}^{\infty} P(U_{n} \ne 0) = \sum_{n=1}^{\infty} P(|X_{n}| > n^{1/p}) = \sum_{n=1}^{\infty} P(|X_{1}| > n^{1/p}) < \infty.
\end{equation*}

\noindent Therefore, by Borel--Cantelli lemma,

\begin{equation}\label{e204}
U_{n} \to 0  \qquad \mbox{a.s.}
\end{equation}

\noindent Moreover

\begin{equation}\label{e205}
\frac{\sum_{i=1}^{n} |X_{i}| \mathbb{I}_{\{|X_{i}| > n^{1/p}\}}}{n^{1/p}} \leqslant \frac{\sum_{i=1}^{n} |X_{i}|^{p} \mathbb{I}_{\{|X_{i}| > n^{1/p}\}}}{n} \leqslant \frac{\sum_{i=1}^{n} |X_{i}|^{p} \mathbb{I}_{\{|X_{i}| > i^{1/p}\}}}{n}.
\end{equation}

\noindent By~\eqref{e204} the right-hand side of~\eqref{e205} converges to zero almost sure and relation~\eqref{e201} follows.
\end{proof}

\begin{lemma}\label{Lem202}
Let $\{X_{n}\}_{n=1}^{\infty}$ be a sequence of identically distributed random variables. If $E|X_{1}|^p < \infty$ where $1 < p < 2$, then

\begin{equation}\label{e206}
\frac{\sum_{i=1}^{n} E(|X_{i}| \mathbb{I}_{\{|X_{i}| > n^{1/p}\}})}{n^{1/p}} \to 0 \qquad (n \to \infty).
\end{equation}

\end{lemma}

\begin{proof}
Note that for any non-negative random variable $\xi$ and ${a>0}$,

\begin{equation*}
E(\xi \mathbb{I}_{\{\xi > a\}}) = a P(\xi > a) + \int_{a}^{\infty} P(\xi > x) \, dx.
\end{equation*}

\noindent Hence

\begin{multline}\label{e208}
\frac{\sum_{i=1}^{n} E(|X_{i}| \mathbb{I}_{\{|X_{i}| > n^{1/p}\}})}{n^{1/p}} = \\ = \frac{\sum_{i=1}^{n} \left(n^{1/p} P(|X_{i}| > n^{1/p}) + \int_{n^{1/p}}^{\infty} P(|X_{i}| > x) \, dx \right)}{n^{1/p}} = \\ = n P(|X_{1}| > n^{1/p}) + n^{\frac{p-1}{p}} \int_{n^{1/p}}^{\infty} P(|X_{1}| > x) \, dx = I_{1n} + I_{2n}.
\end{multline}

\noindent Using~\eqref{e202}, we get

\begin{equation}\label{e209}
I_{1n} = n P(|X_{1}| > n^{1/p}) \to 0 \qquad (n \to \infty).
\end{equation}

\noindent From obvious equality

\begin{equation}\label{e210}
E|X_{1}|^p = p \int_{0}^{\infty} x^{p-1} P(|X_{1}| > x) \, dx
\end{equation}

\noindent it follows that

\begin{equation*}
I_{2n} = n^{\frac{p-1}{p}} \int_{n^{1/p}}^{\infty} P(|X_{1}| > x) \, dx \leqslant \int_{n^{1/p}}^{\infty} x^{p-1} P(|X_{1}| > x) \, dx \to 0 \qquad (n \to \infty),
\end{equation*}

\noindent which, in conjunction with~\eqref{e208} and~\eqref{e209}, proves~\eqref{e206}.
\end{proof}

\begin{lemma}\label{Lem203}
Let $\{X_{n}\}_{n=1}^{\infty}$ be a sequence of identically distributed random variables. If $E|X_{1}|^p < \infty$ where $1 < p < 2$, then

\begin{equation}\label{e212}
\sum_{n=1}^{\infty} \frac{1}{2^{\frac{2 n}{p}}} \sum_{k=1}^{2^{n}} E(|X_{k}|^{2} \mathbb{I}_{\{|X_{k}| \leqslant 2^{\frac{n}{p}}\}}) < \infty.
\end{equation}

\end{lemma}

\begin{proof}
Note that for any non-negative random variable $\xi$ and ${a>0}$,

\begin{equation*}
E(\xi \mathbb{I}_{\{\xi \leqslant a\}}) \leqslant \int_{0}^{a} P(\xi > x) \, dx.
\end{equation*}

\noindent Hence, using~\eqref{e210}, for some positive constants $C$ and $C_{1}$, we obtain

\begin{multline*}
\sum_{n=1}^{\infty} \frac{1}{2^{\frac{2 n}{p}}} \sum_{k=1}^{2^{n}} E(|X_{k}|^{2} \mathbb{I}_{\{|X_{k}| \leqslant 2^{\frac{n}{p}}\}}) \leqslant \\ \leqslant \sum_{n=1}^{\infty} \frac{1}{2^{\frac{2 n}{p}}} \sum_{k=1}^{2^{n}} \int_{0}^{2^{\frac{2 n}{p}}} P(|X_{k}| > x^{1/2}) \, dx \leqslant \\ \leqslant C \sum_{n=1}^{\infty} \frac{1}{2^{\frac{2 n}{p}}} \sum_{k=1}^{2^{n}} \int_{0}^{2^{\frac{n}{p}}} y P(|X_{k}| > y) \, dy \leqslant \\ \leqslant C \sum_{n=1}^{\infty} 2^{\frac{n (p-2)}{p}} \int_{0}^{2^{\frac{n}{p}}} y P(|X_{1}| > y) \, dy \leqslant \\ \leqslant C_{1} + C \sum_{n=1}^{\infty} 2^{\frac{n (p-2)}{p}} \sum_{i=1}^{n} \int_{2^{\frac{i-1}{p}}}^{2^{\frac{i}{p}}} y P(|X_{1}| > y) \, dy \leqslant \\ \leqslant C_{1} + C \sum_{i=1}^{\infty} \int_{2^{\frac{i-1}{p}}}^{2^{\frac{i}{p}}} y P(|X_{1}| > y) \, dy \sum_{n=i}^{\infty} 2^{\frac{n (p-2)}{p}} \leqslant \\ \leqslant C_{1} + C \sum_{i=1}^{\infty} 2^{\frac{i (2-p)}{p}} \int_{2^{\frac{i-1}{p}}}^{2^{\frac{i}{p}}} y^{p-1} P(|X_{1}| > y) \, dy \cdot 2^{\frac{i (p-2)}{p}} \leqslant \\ \leqslant C_{1} + C \sum_{i=1}^{\infty} \int_{2^{\frac{i-1}{p}}}^{2^{\frac{i}{p}}} y^{p-1} P(|X_{1}| > y) \, dy \leqslant \\ \leqslant C_{1} + C \int_{0}^{\infty} y^{p-1} P(|X_{1}| > y) \, dy \leqslant C_{1} + C E|X_{1}|^{p} < \infty,
\end{multline*}

\noindent and~\eqref{e212} follows.
\end{proof}

\bigskip

\begin{proof}[Proof of Theorem~\ref{T1}] Without loss of generality it can be assumed that $EX_{1}=0$. Let

\begin{equation*}\label{e317}
X_{i}^{(n)} = X_{i} \mathbb{I}_{\{|X_{i}| \leqslant n^{1/p}\}}, \qquad i \geqslant 1, \; n \geqslant 1,
\end{equation*}

\begin{equation*}\label{e318}
S_{j}^{(n)} = \sum_{i=1}^{j} X_{i}^{(n)}, \qquad j \geqslant 1, \; n \geqslant 1.
\end{equation*}

\textit{Step 1.} Let us prove that

\begin{equation}\label{e319}
   \frac{S_{n}-S_{n}^{(n)}}{n^{1/p}} \to 0 \qquad \mbox{a.s.}
\end{equation}

\noindent We have

\begin{equation*}
\frac{|S_{n}-S_{n}^{(n)}|}{n^{1/p}} = \frac{|\sum_{i=1}^{n} X_{i} \mathbb{I}_{\{|X_{i}| > n^{1/p}\}}|}{n^{1/p}} \leqslant \frac{\sum_{i=1}^{n} |X_{i}| \mathbb{I}_{\{|X_{i}| > n^{1/p}\}}}{n^{1/p}}.
\end{equation*}

\noindent Application of Lemma~\ref{Lem201} yields to~\eqref{e319}.

\textit{Step 2.} Let us show that

\begin{equation}\label{e321}
   \frac{ES_{n}^{(n)}}{n^{1/p}} \to 0 \qquad (n \to \infty).
\end{equation}

\noindent We have

\begin{multline*}
\frac{|ES_{n}^{(n)}|}{n^{1/p}} = \frac{|\sum_{i=1}^{n} EX_{i}^{(n)}|}{n^{1/p}} \leqslant \frac{\sum_{i=1}^{n} |EX_{i}^{(n)}|}{n^{1/p}} = \\ = \frac{\sum_{i=1}^{n} |E(X_{i}-X_{i}^{(n)})|}{n^{1/p}} \leqslant \frac{\sum_{i=1}^{n} E(|X_{i}| \mathbb{I}_{\{|X_{i}| > n^{1/p}\}})}{n^{1/p}}.
\end{multline*}

\noindent The application of Lemma~\ref{Lem202} yields to~\eqref{e321}.

Now we note that to conclude the proof of the theorem, it is sufficiently to show that

\begin{equation}\label{e323}
   \frac{S_{n}^{(n)}-ES_{n}^{(n)}}{n^{1/p}} \to 0 \qquad \mbox{a.s.}
\end{equation}

\textit{Step 3.} Let us prove that

\begin{equation}\label{e324}
\frac{S_{2^{n}}^{(2^{n})}-ES_{2^{n}}^{(2^{n})}}{2^{\frac{n}{p}}} \to 0 \qquad \mbox{a.s.}
\end{equation}

\noindent Using Lemma~\ref{Lem203}, by Chebyshev's inequality, for any $\varepsilon >0$, we obtain

\begin{multline*}
\sum_{n=1}^{\infty} P \left( \left| \frac{S_{2^{n}}^{(2^{n})}-ES_{2^{n}}^{(2^{n})}}{2^{\frac{n}{p}}} \right| > \varepsilon  \right) \leqslant \frac{1}{\varepsilon^{2}} \sum_{n=1}^{\infty} \frac{Var (S_{2^{n}}^{(2^{n})})}{2^{\frac{2 n}{p}}} = \frac{1}{\varepsilon^{2}} \sum_{n=1}^{\infty} \frac{\sum_{k=1}^{2^{n}} Var (X_{k}^{(2^{n})})}{2^{\frac{2 n}{p}}} \leqslant \\ \leqslant \frac{1}{\varepsilon^{2}} \sum_{n=1}^{\infty} \frac{\sum_{k=1}^{2^{n}} E(X_{k}^{(2^{n})})^{2}}{2^{\frac{2 n}{p}}} = \frac{1}{\varepsilon^{2}} \sum_{n=1}^{\infty} \frac{1}{2^{\frac{2 n}{p}}} \sum_{k=1}^{2^{n}} E(|X_{k}|^{2} \mathbb{I}_{\{|X_{k}| \leqslant 2^{\frac{n}{p}}\}}) < \infty.
\end{multline*}

\noindent Thus, by Borel-Cantelli lemma, we have that relation~\eqref{e324} is proved.

\textit{Step 4.} Let us prove that

\begin{equation}\label{e328}
\lim_{n \to \infty} \max_{2^{n}<k \leqslant 2^{n+1}} \left| \frac{\sum\limits_{i=2^{n}+1}^{k} (X_{i}^{(i)} - EX_{i}^{(i)})}{2^{\frac{n+1}{p}}} \right| = 0 \qquad \mbox{a.s.}
\end{equation}

\noindent Using Lemma~\ref{Lem203}, by Chebyshev's inequality, for any $\varepsilon >0$, we obtain

\begin{multline*}
\sum_{n=1}^{\infty} P \left(\max_{2^{n}<k \leqslant 2^{n+1}} \left| \frac{\sum\limits_{i=2^{n}+1}^{k} (X_{i}^{(i)} - EX_{i}^{(i)})}{2^{\frac{n+1}{p}}} \right| > \varepsilon  \right) \leqslant \\ \leqslant \frac{1}{\varepsilon^{2}} \sum_{n=1}^{\infty} \frac{1}{2^{\frac{2 (n+1)}{p}}} \sum_{k=1}^{2^{n+1}} E(|X_{k}|^{2} \mathbb{I}_{\{|X_{k}| \leqslant 2^{\frac{n+1}{p}}\}}) < \infty.
\end{multline*}

\noindent The application of Borel-Cantelli lemma yields to~\eqref{e328}.

\textit{Step 5.} We shall prove that

\begin{equation}\label{e332}
\lim_{n \to \infty} \max_{2^{n}<k \leqslant 2^{n+1}} \frac{ \left| S_{k}^{(2^{n})} - ES_{k}^{(2^{n})} \right| }{2^{\frac{n+1}{p}}} = 0 \qquad \mbox{a.s.}
\end{equation}

\noindent For $n \geqslant 1$ and $k$ such that $2^{n}<k \leqslant 2^{n+1}$ we have

\begin{multline*}
\left| S_{k}^{(2^{n})} - ES_{k}^{(2^{n})} \right| = \left| S_{k}^{(2^{n})} - ES_{k}^{(2^{n})} + S_{2^{n}}^{(2^{n})} - S_{2^{n}}^{(2^{n})} + ES_{2^{n}}^{(2^{n})} - ES_{2^{n}}^{(2^{n})} \right| = \\ = \left| (S_{k}^{(2^{n})} - S_{2^{n}}^{(2^{n})}) - E(S_{k}^{(2^{n})} - S_{2^{n}}^{(2^{n})}) + (S_{2^{n}}^{(2^{n})} - ES_{2^{n}}^{(2^{n})}) \right| \leqslant \\ \leqslant \left| (S_{k}^{(2^{n})} - S_{2^{n}}^{(2^{n})}) - E(S_{k}^{(2^{n})} - S_{2^{n}}^{(2^{n})}) \right| + \left| (S_{2^{n}}^{(2^{n})} - ES_{2^{n}}^{(2^{n})}) \right| = \\ = \left| \sum_{i=2^{n}+1}^{k} X_{i} \mathbb{I}_{\{|X_{i}| \leqslant 2^{\frac{n}{p}}\}} - E(\sum_{i=2^{n}+1}^{k} X_{i} \mathbb{I}_{\{|X_{i}| \leqslant 2^{\frac{n}{p}}\}}) \right| + \left| (S_{2^{n}}^{(2^{n})} - ES_{2^{n}}^{(2^{n})}) \right| = \\ = |\sum_{i=2^{n}+1}^{k} (X_{i} \mathbb{I}_{\{|X_{i}| \leqslant i^{1/p}\}} - X_{i} \mathbb{I}_{\{2^{\frac{n}{p}} < |X_{i}| \leqslant i^{1/p}\}}) - \\ - E ( \sum_{i=2^{n}+1}^{k} (X_{i} \mathbb{I}_{\{|X_{i}| \leqslant i^{1/p}\}} - X_{i} \mathbb{I}_{\{2^{\frac{n}{p}} < |X_{i}| \leqslant i^{1/p}\}}))| + \left| (S_{2^{n}}^{(2^{n})} - ES_{2^{n}}^{(2^{n})}) \right| = \\ = |\sum_{i=2^{n}+1}^{k} (X_{i} \mathbb{I}_{\{|X_{i}| \leqslant i^{1/p}\}} - E(X_{i} \mathbb{I}_{\{|X_{i}| \leqslant i^{1/p}\}})) - \\ - \sum_{i=2^{n}+1}^{k} (X_{i} \mathbb{I}_{\{2^{\frac{n}{p}} < |X_{i}| \leqslant i^{1/p}\}} - E(X_{i} \mathbb{I}_{\{2^{\frac{n}{p}} < |X_{i}| \leqslant i^{1/p}\}}))| + \left| (S_{2^{n}}^{(2^{n})} - ES_{2^{n}}^{(2^{n})}) \right| = \\ = |\sum_{i=2^{n}+1}^{k} (X_{i}^{(i)} - EX_{i}^{(i)}) - \sum_{i=2^{n}+1}^{k} X_{i} \mathbb{I}_{\{2^{\frac{n}{p}} < |X_{i}| \leqslant i^{1/p}\}} + \sum_{i=2^{n}+1}^{k} E(X_{i} \mathbb{I}_{\{2^{\frac{n}{p}} < |X_{i}| \leqslant i^{1/p}\}})| + \\ + \left| (S_{2^{n}}^{(2^{n})} - ES_{2^{n}}^{(2^{n})}) \right| \leqslant \left| \sum_{i=2^{n}+1}^{k} (X_{i}^{(i)} - EX_{i}^{(i)}) \right| + \left| \sum_{i=2^{n}+1}^{k} X_{i} \mathbb{I}_{\{2^{\frac{n}{p}} < |X_{i}| \leqslant i^{1/p}\}} \right| + \\ + \left| \sum_{i=2^{n}+1}^{k} E(X_{i} \mathbb{I}_{\{2^{\frac{n}{p}} < |X_{i}| \leqslant i^{1/p}\}}) \right| + \left| (S_{2^{n}}^{(2^{n})} - ES_{2^{n}}^{(2^{n})}) \right| \leqslant \\ \leqslant \left| \sum_{i=2^{n}+1}^{k} (X_{i}^{(i)} - EX_{i}^{(i)}) \right| + \sum_{i=2^{n}+1}^{k} |X_{i}| \mathbb{I}_{\{2^{\frac{n}{p}} < |X_{i}| \leqslant i^{1/p}\}} + \\ + \sum_{i=2^{n}+1}^{k} E(|X_{i}| \mathbb{I}_{\{|X_{i}| > 2^{\frac{n}{p}}\}}) + \left| (S_{2^{n}}^{(2^{n})} - ES_{2^{n}}^{(2^{n})}) \right| \leqslant \\ \leqslant \left| \sum_{i=2^{n}+1}^{k} (X_{i}^{(i)} - EX_{i}^{(i)}) \right| + \sum_{i=2^{n}+1}^{2^{n+1}} |X_{i}| \mathbb{I}_{\{2^{\frac{n}{p}} < |X_{i}| \leqslant i^{1/p}\}} + \\ + \sum_{i=2^{n}+1}^{2^{n+1}} E(|X_{i}| \mathbb{I}_{\{|X_{i}| > 2^{\frac{n}{p}}\}}) + \left| (S_{2^{n}}^{(2^{n})} - ES_{2^{n}}^{(2^{n})}) \right| \leqslant \\ \leqslant \left| \sum_{i=2^{n}+1}^{k} (X_{i}^{(i)} - EX_{i}^{(i)}) \right| + \sum_{i=1}^{2^{n+1}} |X_{i}| \mathbb{I}_{\{|X_{i}| > 2^{\frac{n}{p}}\}} + \\ + \sum_{i=1}^{2^{n}} E(|X_{i}| \mathbb{I}_{\{|X_{i}| > 2^{\frac{n}{p}}\}}) + \left| (S_{2^{n}}^{(2^{n})} - ES_{2^{n}}^{(2^{n})}) \right|.
\end{multline*}

\noindent Therefore

\begin{multline*}
\max_{2^{n}<k \leqslant 2^{n+1}} \left| S_{k}^{(2^{n})} - ES_{k}^{(2^{n})} \right| \leqslant \max_{2^{n}<k \leqslant 2^{n+1}} \left| \sum_{i=2^{n}+1}^{k} (X_{i}^{(i)} - EX_{i}^{(i)}) \right| + \\ + \sum_{i=1}^{2^{n+1}} |X_{i}| \mathbb{I}_{\{|X_{i}| > 2^{\frac{n}{p}}\}} + \sum_{i=1}^{2^{n}} E(|X_{i}| \mathbb{I}_{\{|X_{i}| > 2^{\frac{n}{p}}\}}) + \left| S_{2^{n}}^{(2^{n})} - ES_{2^{n}}^{(2^{n})} \right|.
\end{multline*}

\noindent The application of Lemmas~\ref{Lem201} and~\ref{Lem202} and relations~\eqref{e324} and~\eqref{e328} yields to~\eqref{e332}.

\textit{Step 6.} We shall prove that

\begin{equation}\label{e335}
\lim_{n \to \infty} \max_{2^{n}<k \leqslant 2^{n+1}} \frac{ \left| S_{2^{n}}^{(k)} - ES_{2^{n}}^{(k)} \right| }{2^{\frac{n+1}{p}}} = 0 \qquad \mbox{a.s.}
\end{equation}

\noindent For $n \geqslant 1$ and $k$ such that $2^{n}<k \leqslant 2^{n+1}$ we have

\begin{multline*}
\left| S_{2^{n}}^{(k)} - ES_{2^{n}}^{(k)} \right| = \left| S_{2^{n}}^{(k)} - ES_{2^{n}}^{(k)} + S_{2^{n}}^{(2^{n})} - S_{2^{n}}^{(2^{n})} + ES_{2^{n}}^{(2^{n})} - ES_{2^{n}}^{(2^{n})} \right| = \\ = \left| (S_{2^{n}}^{(k)} - S_{2^{n}}^{(2^{n})}) - E(S_{2^{n}}^{(k)} - S_{2^{n}}^{(2^{n})}) + (S_{2^{n}}^{(2^{n})} - ES_{2^{n}}^{(2^{n})}) \right| \leqslant \\ \leqslant \left| (S_{2^{n}}^{(k)} - S_{2^{n}}^{(2^{n})}) - E(S_{2^{n}}^{(k)} - S_{2^{n}}^{(2^{n})}) \right| + \left| (S_{2^{n}}^{(2^{n})} - ES_{2^{n}}^{(2^{n})}) \right| = \\ = |\sum_{i=1}^{2^{n}} (X_{i} \mathbb{I}_{\{|X_{i}| \leqslant k^{1/p}\}} - X_{i} \mathbb{I}_{\{|X_{i}| \leqslant 2^{\frac{n}{p}}\}}) - E ( \sum_{i=1}^{2^{n}} (X_{i} \mathbb{I}_{\{|X_{i}| \leqslant k^{1/p}\}} - X_{i} \mathbb{I}_{\{|X_{i}| \leqslant 2^{\frac{n}{p}}\}}))| + \\ + \left| (S_{2^{n}}^{(2^{n})} - ES_{2^{n}}^{(2^{n})}) \right| = \\ = |\sum_{i=1}^{2^{n}} X_{i} \mathbb{I}_{\{2^{\frac{n}{p}} < |X_{i}| \leqslant k^{1/p}\}} - \sum_{i=1}^{2^{n}} E(X_{i} \mathbb{I}_{\{2^{\frac{n}{p}} < |X_{i}| \leqslant k^{1/p}\}})| + \left| (S_{2^{n}}^{(2^{n})} - ES_{2^{n}}^{(2^{n})}) \right| \leqslant \\ \leqslant |\sum_{i=1}^{2^{n}} X_{i} \mathbb{I}_{\{2^{\frac{n}{p}} < |X_{i}| \leqslant k^{1/p}\}}| + |\sum_{i=1}^{2^{n}} E(X_{i} \mathbb{I}_{\{2^{\frac{n}{p}} < |X_{i}| \leqslant k^{1/p}\}})| + \left| (S_{2^{n}}^{(2^{n})} - ES_{2^{n}}^{(2^{n})}) \right| \leqslant \\ \leqslant \sum_{i=1}^{2^{n}} |X_{i}| \mathbb{I}_{\{2^{\frac{n}{p}} < |X_{i}| \leqslant k^{1/p}\}} + \sum_{i=1}^{2^{n}} E(|X_{i}| \mathbb{I}_{\{2^{\frac{n}{p}} < |X_{i}| \leqslant k^{1/p}\}}) + \left| (S_{2^{n}}^{(2^{n})} - ES_{2^{n}}^{(2^{n})}) \right| \leqslant \\ \leqslant \sum_{i=1}^{2^{n}} |X_{i}| \mathbb{I}_{\{|X_{i}| > 2^{\frac{n}{p}}\}} + \sum_{i=1}^{2^{n}} E(|X_{i}| \mathbb{I}_{\{|X_{i}| > 2^{\frac{n}{p}}\}}) + \left| (S_{2^{n}}^{(2^{n})} - ES_{2^{n}}^{(2^{n})}) \right|.
\end{multline*}

\noindent Therefore

\begin{multline*}
\max_{2^{n}<k \leqslant 2^{n+1}} \left| S_{2^{n}}^{(k)} - ES_{2^{n}}^{(k)} \right| \leqslant \sum_{i=1}^{2^{n}} |X_{i}| \mathbb{I}_{\{|X_{i}| > 2^{\frac{n}{p}}\}} + \\ + \sum_{i=1}^{2^{n}} E(|X_{i}| \mathbb{I}_{\{|X_{i}| > 2^{\frac{n}{p}}\}}) + \left| S_{2^{n}}^{(2^{n})} - ES_{2^{n}}^{(2^{n})} \right|.
\end{multline*}

\noindent The application of Lemmas~\ref{Lem201} and~\ref{Lem202} and relation~\eqref{e324} yields to~\eqref{e335}.

\textit{Step 7.} We shall prove that

\begin{equation}\label{e338}
\lim_{n \to \infty} \max_{2^{n}<k \leqslant 2^{n+1}} \left| \frac{S_{k}^{(k)} - ES_{k}^{(k)}}{2^{\frac{n+1}{p}}} \right| = 0 \qquad \mbox{a.s.}
\end{equation}

\noindent For $n \geqslant 1$ and $k$ such that $2^{n}<k \leqslant 2^{n+1}$ we have

\begin{multline*}
\left| S_{k}^{(k)} - ES_{k}^{(k)} \right| = \\ = | \left[ (S_{k}^{(k)} - S_{2^{n}}^{(k)}) + (S_{k}^{(2^{n})} - S_{2^{n}}^{(2^{n})}) \right] - E \left[ (S_{k}^{(k)} - S_{2^{n}}^{(k)}) + (S_{k}^{(2^{n})} - S_{2^{n}}^{(2^{n})}) \right] + \\ + (S_{2^{n}}^{(k)} - ES_{2^{n}}^{(k)}) - (S_{k}^{(2^{n})} - ES_{k}^{(2^{n})}) + (S_{2^{n}}^{(2^{n})} - ES_{2^{n}}^{(2^{n})})| \leqslant \\ \leqslant \left| \sum_{i=2^{n}+1}^{k} X_{i} \mathbb{I}_{\{2^{\frac{n}{p}} < |X_{i}| \leqslant k^{1/p}\}} - E (\sum_{i=2^{n}+1}^{k} X_{i} \mathbb{I}_{\{2^{\frac{n}{p}} < |X_{i}| \leqslant k^{1/p}\}}) \right| + \\ + \left| S_{2^{n}}^{(k)} - ES_{2^{n}}^{(k)} \right| + \left| S_{k}^{(2^{n})} - ES_{k}^{(2^{n})} \right| + \left| S_{2^{n}}^{(2^{n})} - ES_{2^{n}}^{(2^{n})} \right| \leqslant \\ \leqslant \left| \sum_{i=2^{n}+1}^{k} X_{i} \mathbb{I}_{\{2^{\frac{n}{p}} < |X_{i}| \leqslant k^{1/p}\}} \right| + \left| \sum_{i=2^{n}+1}^{k} E(X_{i} \mathbb{I}_{\{2^{\frac{n}{p}} < |X_{i}| \leqslant k^{1/p}\}}) \right| + \\ + \left| S_{2^{n}}^{(k)} - ES_{2^{n}}^{(k)} \right| + \left| S_{k}^{(2^{n})} - ES_{k}^{(2^{n})} \right| + \left| S_{2^{n}}^{(2^{n})} - ES_{2^{n}}^{(2^{n})} \right| \leqslant \\ \leqslant \sum_{i=2^{n}+1}^{2^{n+1}} |X_{i}| \mathbb{I}_{\{2^{\frac{n}{p}} < |X_{i}| \leqslant 2^{\frac{n+1}{p}}\}} + \sum_{i=2^{n}+1}^{2^{n+1}} E(|X_{i}| \mathbb{I}_{\{2^{\frac{n}{p}} < |X_{i}| \leqslant 2^{\frac{n+1}{p}}\}}) + \\ + \left| S_{2^{n}}^{(k)} - ES_{2^{n}}^{(k)} \right| + \left| S_{k}^{(2^{n})} - ES_{k}^{(2^{n})} \right| + \left| S_{2^{n}}^{(2^{n})} - ES_{2^{n}}^{(2^{n})} \right| \leqslant \\ \leqslant \sum_{i=1}^{2^{n+1}} |X_{i}| \mathbb{I}_{\{|X_{i}| > 2^{\frac{n}{p}}\}} + \sum_{i=1}^{2^{n}} E(|X_{i}| \mathbb{I}_{\{|X_{i}| > 2^{\frac{n}{p}}\}}) + \\ + \left| S_{2^{n}}^{(k)} - ES_{2^{n}}^{(k)} \right| + \left| S_{k}^{(2^{n})} - ES_{k}^{(2^{n})} \right| + \left| S_{2^{n}}^{(2^{n})} - ES_{2^{n}}^{(2^{n})} \right|.
\end{multline*}

\noindent Therefore

\begin{multline*}
\max_{2^{n} < k \leqslant 2^{n+1}} \left| S_{k}^{(k)} - ES_{k}^{(k)} \right| \leqslant \sum_{i=1}^{2^{n+1}} |X_{i}| \mathbb{I}_{\{|X_{i}| > 2^{\frac{n}{p}}\}} + \\ + \sum_{i=1}^{2^{n}} E(|X_{i}| \mathbb{I}_{\{|X_{i}| > 2^{\frac{n}{p}}\}}) + \max_{2^{n} < k \leqslant 2^{n+1}} \left| S_{2^{n}}^{(k)} - ES_{2^{n}}^{(k)} \right| + \\ + \max_{2^{n} < k \leqslant 2^{n+1}} \left| S_{k}^{(2^{n})} - ES_{k}^{(2^{n})} \right| + \left| S_{2^{n}}^{(2^{n})} - ES_{2^{n}}^{(2^{n})} \right|.
\end{multline*}

\noindent The application of Lemmas~\ref{Lem201} and~\ref{Lem202} and relations~\eqref{e324},~\eqref{e332} and~\eqref{e335} yields to~\eqref{e338}. Relation~\eqref{e338} implies~\eqref{e323}. Relations~\eqref{e319},~\eqref{e321} and~\eqref{e323} imply~\eqref{e1001}. Theorem~\ref{T1} is proved.
\end{proof}

\end{document}